\documentclass{amsart}
\usepackage{amsmath, amscd}
\usepackage{amsfonts}

\begin{document}
\numberwithin{equation}{section}

\def\1#1{\overline{#1}}
\def\2#1{\widetilde{#1}}
\def\3#1{\widehat{#1}}
\def\4#1{\mathbb{#1}}
\def\5#1{\frak{#1}}
\def\6#1{{\mathcal{#1}}}

\newcommand{\w}{\omega}
\newcommand{\Lie}[1]{\ensuremath{\mathfrak{#1}}}
\newcommand{\LieL}{\Lie{l}}
\newcommand{\LieH}{\Lie{h}}
\newcommand{\LieG}{\Lie{g}}
\newcommand{\de}{\partial}
\newcommand{\R}{\mathbb R}
\newcommand{\Q}{\mathbb Q}
\newcommand{\FH}{{\sf Fix}(H_p)}
\newcommand{\al}{\alpha}
\newcommand{\tr}{\widetilde{\rho}}
\newcommand{\tz}{\widetilde{\zeta}}
\newcommand{\tk}{\widetilde{C}}
\newcommand{\tv}{\widetilde{\varphi}}
\newcommand{\hv}{\hat{\varphi}}
\newcommand{\tu}{\tilde{u}}
\newcommand{\tF}{\tilde{F}}
\newcommand{\debar}{\overline{\de}}
\newcommand{\Z}{\mathbb Z}
\newcommand{\C}{\mathbb C}
\newcommand{\Po}{\mathbb P}
\newcommand{\zbar}{\overline{z}}
\newcommand{\G}{\mathcal{G}}
\newcommand{\So}{\mathcal{S}}
\newcommand{\Ko}{\mathcal{K}}
\newcommand{\U}{\mathcal{U}}
\newcommand{\B}{\mathbb B}
\newcommand{\NC}{\mathcal N\mathcal C}
\newcommand{\oB}{\overline{\mathbb B}}
\newcommand{\Cur}{\mathcal D}
\newcommand{\Dis}{\mathcal Dis}
\newcommand{\Levi}{\mathcal L}
\newcommand{\SP}{\mathcal SP}
\newcommand{\Sp}{\mathcal Q}
\newcommand{\A}{\mathcal O^{k+\alpha}(\overline{\mathbb D},\C^n)}
\newcommand{\CA}{\mathcal C^{k+\alpha}(\de{\mathbb D},\C^n)}
\newcommand{\Ma}{\mathcal M}
\newcommand{\Ac}{\mathcal O^{k+\alpha}(\overline{\mathbb D},\C^{n}\times\C^{n-1})}
\newcommand{\Acc}{\mathcal O^{k-1+\alpha}(\overline{\mathbb D},\C)}
\newcommand{\Acr}{\mathcal O^{k+\alpha}(\overline{\mathbb D},\R^{n})}
\newcommand{\Co}{\mathcal C}
\newcommand{\Hol}{{\sf Hol}(\mathbb H, \mathbb C)}
\newcommand{\Aut}{{\sf Aut}(\mathbb D)}
\newcommand{\D}{\mathbb D}
\newcommand{\oD}{\overline{\mathbb D}}
\newcommand{\oX}{\overline{X}}
\newcommand{\loc}{L^1_{\rm{loc}}}
\newcommand{\la}{\langle}
\newcommand{\ra}{\rangle}
\newcommand{\thh}{\tilde{h}}
\newcommand{\N}{\mathbb N}
\newcommand{\kd}{\kappa_D}
\newcommand{\Ha}{\mathbb H}
\newcommand{\ps}{{\sf Psh}}
\newcommand{\Hess}{{\sf Hess}}
\newcommand{\subh}{{\sf subh}}
\newcommand{\harm}{{\sf harm}}
\newcommand{\Klim}{K\textrm{-}\lim\limits}
\newcommand{\ph}{{\sf Ph}}
\newcommand{\tl}{\tilde{\lambda}}
\newcommand{\gdot}{\stackrel{\cdot}{g}}
\newcommand{\gddot}{\stackrel{\cdot\cdot}{g}}
\newcommand{\fdot}{\stackrel{\cdot}{f}}
\newcommand{\fddot}{\stackrel{\cdot\cdot}{f}}
\def\v{\varphi}
\def\Re{{\sf Re}\,}
\def\Im{{\sf Im}\,}
\def\rk{{\rm rank\,}}
\def\rg{{\sf rg}\,}
\def\Gen{{\sf Gen}(\D)}
\def\Pl{\mathcal P}
\def\br{{\sf BRFP}}

\newtheorem{theorem}{Theorem}[section]
\newtheorem{lemma}[theorem]{Lemma}
\newtheorem{proposition}[theorem]{Proposition}
\newtheorem{corollary}[theorem]{Corollary}

\theoremstyle{definition}
\newtheorem{definition}[theorem]{Definition}
\newtheorem{example}[theorem]{Example}

\theoremstyle{remark}
\newtheorem{remark}[theorem]{Remark}
\numberwithin{equation}{section}

\title[Common fixed points]{Common boundary regular fixed points for holomorphic semigroups in strongly convex domains}
\author[M. Abate]{Marco Abate}
\address{M. Abate: Dipartimento di Matematica, Universit\`a di Pisa, L.go Pontecorvo 5,
56127 Pisa, Italy} \email{abate@dm.unipi.it}
\author[F. Bracci]{Filippo Bracci$^\dag$}
\address{F. Bracci: Dipartimento Di Matematica, Universit\`{a} Di Roma \textquotedblleft Tor
Vergata\textquotedblright, Via Della Ricerca Scientifica 1,
00133, Roma, Italy. } \email{fbracci@mat.uniroma2.it}
\thanks{$^{\dag}$Supported by the ERC grant ``HEVO - Holomorphic Evolution Equations'' n. 277691.}

\dedicatory{Dedicated to Prof. David Shoikhet for his 60th anniversary}

\begin{abstract}
Let $D$ be a bounded strongly convex domain with smooth boundary in $\C^N$. Let $(\phi_t)$ be a continuous semigroup of holomorphic self-maps of $D$. We prove that if $p\in \de D$ is an isolated boundary regular fixed point for  $\phi_{t_0}$ for some $t_0>0$, then $p$ is a boundary regular fixed point for $\phi_t$ for all~$t\geq 0$. Along the way we also study backward iteration sequences for elliptic holomorphic self-maps of~$D$.
\end{abstract}

\maketitle

\section{Introduction}

Let $D\subset \C^N$ be a bounded domain.  A continuous one-parameter semigroup of holomorphic self-maps of $D$ (or, shortly a semigroup of holomorphic self-maps of $D$) is a family  $(\phi_t)_{t\geq 0}$ of holomorphic self-maps of $D$ such that $\phi_{t+s}=\phi_t\circ \phi_s$ for all $s,t\geq 0$, $\phi_0={\sf id}_D$ and  $\R^+\ni t\mapsto \phi_t(z)$ is locally absolutely continuous locally uniformly in $z$. Namely, $(\phi_t)$ is a continuous semigroup morphism between the semigroup  $(\R^+,+)$ endowed with the Euclidean topology and the composition semigroup of holomorphic self-maps of $D$ endowed with the topology of uniform convergence on compacta. Every semigroup of holomorphic self-maps of $D$ is generated by a $\R$-semi-complete holomorphic vector field on $D$, called the infinitesimal generator associated with the semigroup.

Semigroups of holomorphic self-maps have been extensively studied (see, {\sl e.g.} \cite{RS}), in connection  with various areas of analysis, including geometric function theory, operator theory, iteration theory, theory of branching stochastic processes, Loewner theory.

The aim of this paper is to give a contribution to boundary dynamics of semigroups on a bounded strongly convex domain with smooth boundary $D\subset\C^n$, studying common boundary (regular) fixed points of semigroups in $D$ (among which, a particular interesting case is the unit ball of $\C^N$).

As holomorphic self-maps of a domain might not extend continuously to the boundary, if $p\in \de D$ and $f\colon D\to D$ is holomorphic, one might think of $p$ as a boundary fixed point of $f$ if $f$ admits limit $p$ along all sequences converging to $p$ in some  ``admissible'' subset of $D$ (see Section \ref{brfpsec} for precise definitions). In case $D=\D$ the unit disc of $\C$, such admissible subsets are exactly the Stolz angles, and thus $p\in \de \D$ is a boundary fixed point of $f$ if $f$ has non-tangential limit $p$ at $p$.

Boundary fixed points can be divided into two categories, the {\sl boundary regular fixed points}---BRFP's for short---and the {\sl irregular} (or super-repulsive) fixed points. In the unit disc, the first category is formed by those boundary fixed points for which the angular derivative of the map exists finitely. By the classical Julia-Wolff-Carath\'eodory theorem, the existence of the (finite) angular derivative at a boundary point for a holomorphic self-map of $\D$ corresponds to the finiteness of the so-called {\sl boundary dilatation coefficient}---which, roughly speaking,  measure the rate of approach of  $f(z)$ to~$p$ as $z\to p$. In higher dimension, a boundary regular fixed point for a holomorphic self-map $f$ of  $D$ is a point $p\in \de D$ for which the admissible limit of $f$ at $p$ is $p$ and the boundary dilation coefficient of $f$ at $p$ is finite.

{\sl Common} boundary regular fixed points for a semigroup ({\sl i.e.}, boundary points which are BRFP's for each element of the semigroup) have been studied and characterized in terms of the local behavior of the associated infinitesimal generator in \cite{CDP2, E-S2, E-R-S, ERS06, AS, BCD}.

Moreover, it is known (see, \cite[Theorem 1]{CDP}, \cite[Theorem 2]{CDP2}, \cite[pag. 255]{Siskakis-tesis}, \cite{ES}) that given a semigroup $(\phi_t)$ of holomorphic self-maps of $\D$, a point $p\in \de \D$ is a boundary (regular) fixed point of $\phi_{t_0}$ for some $t_0>0$ if and only if it is a boundary (regular) fixed point of $\phi_t$ for all $t\geq 0$. The proof of this fact relies on the existence of the so-called K\"{o}nigs function, a univalent map from $\D$ to $\C$ which (simultaneously) linearizes the semigroup $(\phi_t)$. Such a tool is not available in higher dimension in general, due essentially to the lack of a Riemann uniformization theorem.

The aim of this paper is to extend part of the previous results to higher dimension  using an argument based on complex geodesics and backward iteration sequences. In order to state our result, we say that if $f\colon D\to D$ is holomorphic and $p\in \de D$ is a BRFP with boundary dilatation coefficient $A>0$, then $p$ is {\sl isolated} if there exists a neighborhood $U$ of $p$ such that $U\cap \de D$ contains no BRFP's for $f$ with boundary dilatation coefficients $\leq A$ except $p$.

Then our main result (proved in Section \ref{prova}) can be stated as follows

\begin{theorem}\label{main}
Let $D\subset \C^N$ be a bounded strongly convex domain with smooth boundary. Let $(\phi_t)$ be a semigroup of holomorphic self-maps of $D$. Suppose $p\in \de D$    is an isolated boundary regular fixed point for $\phi_{t_0}$ for some $t_0>0$. Then $p$ is a boundary regular fixed point for $\phi_t$ for all $t\geq  0$.
\end{theorem}

The proof of Theorem \ref{main} relies on the study of ``regular contact points'' for semigroups and backward iteration sequences. In particular, in Section \ref{Scontact},  using complex geodesics, we prove that given a regular contact point $p\in \de D$ for $\phi_{t_0}$ for some $t_0>0$, then $p$ is a regular contact point for $\phi_t$ for all $t\in [0,t_0]$ and the curve $[0,t_0]\ni t\mapsto \phi_t(p)\in \de D$ is continuous, extending to higher dimension one of the results in \cite{BP}. Along the way, in Section \ref{Sback} we extend the results about existence and convergence of backward iteration sequences for holomorphic self-maps of $D$ obtained in \cite{AR1} to the case of ``rotational elliptic maps''.

\section{BRFP's in strongly convex domains}\label{brfpsec}

Let $D$ be a bounded strongly convex domain in $\C^n$ with smooth
boundary.  A {\sl complex geodesic} is a holomorphic map $\v\colon \D\to
D$ which is an isometry between the Poincar\'e metric $k_{\D}$ of
$\D=\{\zeta\in \C\mid |\zeta|<1\}$ and the Kobayashi distance $k_D$ in
$D$ (\cite{Kob}). A holomorphic map $h\colon  \D \to D$ is a complex geodesic if and only if it is an infinitesimal isometry between the Poincar\'e metric $\kappa_\D$ of $\D$ and the Kobayashi metric $\kappa_D$ of $D$ (see \cite[Ch. 2.6]{A}).

According to  Lempert (see \cite{Le, Le1, Le2} and \cite{A}), any complex
geodesic extends smoothly to the boundary of the disc and $\v(\de
\D)\subset \de D$. Moreover,  given any two points $z,w\in
\overline{D}$, $z\neq w$, there exists a complex geodesic $\v\colon \D\to
D$ such that $z,w\in \v(\oD)$. Such a geodesic is unique up to
pre-composition with automorphisms of $\D$. Conversely, if $\v\colon \D \to D$ is a holomorphic map such that $k_D(\v(\zeta_1), \v(\zeta_2))=k_{\D}(\zeta_1,\zeta_2)$ for some $\zeta_1\neq \zeta_2\in \D$, then $\v$ is a complex geodesic.

Similarly, given $z\in D$ and $v\in T_zD\setminus \{0\}$, there exists a unique  complex geodesic such that $\v(0)=z$ and $\v'(0)=\lambda v$ for some $\lambda>0$.

If $\v\colon \D\to D$ is a complex geodesic then there exists a (unique when suitably normalized)
holomorphic map $\tilde{\rho}\colon D\to \D$, smooth up to  $\de
D$ such that $\tilde{\rho} \circ \varphi = {\sf id}_{\D}$. The map $\tilde{\rho}$ is called the {\em left inverse } of $\varphi$. It is known that $\tilde{\rho}^{-1}(e^{i\theta})=\{\v(e^{i\theta})\}$ for all $\theta\in \R$, while the fibers $\tr^{-1}(\zeta)$ are the intersection of $D$ with affine complex hyperplanes  for all $\zeta\in \D$ (see, {\sl e.g.}, \cite[Section 3]{BPT}).

In the sequel we shall use the following result (see \cite[Corollary 2.3, Lemma 3.5]{BPT})

\begin{proposition}\label{geo-converge}
Let $D\subset \C^N$ be a bounded strongly convex domain with smooth boundary and let $z_0\in D$. Let $\{\v_k\}_{k\in \N}$ be a family of complex geodesics of $D$ such that $\v_k(0)=z_0$ for all $k\in \N$ and let $\tilde{\rho}_k$ denote the left inverse of $\v_k$ for $k\in \N$. If $\{\v_k\}$ converges uniformly on compacta of $\D$ to a function $\v\colon \D\to \C^N$, then $\v$ is a complex geodesic  and $\v_k\to \v$ uniformly in $\overline{\D}$. Moreover, $\{\tilde{\rho}_k\}$ converges uniformly in $\overline{D}$ to the left inverse $\tilde{\rho}$ of $\v$.
\end{proposition}

Given $z_0\in D$ and $p\in \de D$, we will denote by $\v_p\colon \D \to D$ the unique complex geodesic such that $\v_p(0)=z_0$ and $\v_p(1)=p$ and by $\tilde{\rho}_p$ its left inverse. We will also denote
\[
\rho_p:=\v_p \circ \tilde{\rho}_p \colon  D \to \v(\D).
\]

We recall now the notion of ``admissible limits'' in strongly convex domain (see \cite{A, Ab1}).

\begin{definition}
Let $D\subset \C^N$ be a bounded strongly convex domain with smooth boundary and let $z_0\in D$ and $p\in \de D$.   A sequence $\{z_k\}\subset D$ converging to $p$ is said to be {\sl special} if
\[
\lim_{k\to \infty} k_D(z_k, \rho_p(z_k))=0.
\]
The sequence $\{z_k\}$ is called {\sl restricted} if $\tilde{\rho}_p(z_k)\to 1$ non-tangentially in $\D$. A  continuous curve $\gamma\colon  [0,1]\to D$ such that $\gamma(1)=p$ is called {\sl special}, respectively {\sl restricted}, provided for any sequence $\{t_k\}\subset [0,1)$ converging to $1$, the sequence $\{\gamma(t_k)\}$ is special, respectively restricted.
\end{definition}

\begin{definition}[\cite{A}, \cite{Ab1}]
Let $D\subset \C^N$ be a bounded strongly convex   domain with smooth boundary and let $z_0\in D$. Let $p \in \partial D$ and $M>1$. The {\sl $K$-region $K_{z_0}(p,M)$} of center
$\tau$, amplitude $M$ and pole $z_0$ is
\[ K_{z_0}(p,M):= \left\{ z \in D \mid \lim_{w \rightarrow p} \left[ k_D(z,w) -
k_D(z_0,w) \right] + k_D(z_0,z) < \log M \right\}. \]
\end{definition}

Let $\{z_k\}\subset D$ be a sequence converging to $p\in \de D$. It is known (\cite[Lemma 2.7.12]{A}, \cite{Ab1}) that if $\{z_k\}\subset K_{z_0}(p,M)$ for some $M>1$ then $\{z_k\}$ is restricted. While, if $\{z_k\}$ is special and restricted then it is eventually contained in a $K$-region $K_{z_0}(p,M)$ for some $M>1$.

A holomorphic self-map $f$ of $D$ has {\sl $K$-limit} $q$ at $p \in \partial D$---and we write $\Klim_{z\to p}f(z)=q$ or, for short, $q=f(p)$---if $\lim_{k\to \infty}f(w_k)=q$   for every sequence converging to $p$ and such that
$\left\{ w_k \right\}\subset K_{z_0}(p,M)$ for some $M>1$.

\begin{definition}
Let $D\subset \C^N$ be a bounded strongly convex   domain with smooth boundary and let $z_0\in D$. Let $f\colon D\to D$ be holomorphic and let $p\in \de D$. The {\sl boundary dilation coefficient $\al_p(f)\in (0,+\infty]$ of $f$ at $p$} is defined by
\[
\frac{1}{2}\log \al_p(f):=\liminf_{w\to p} [k_D(z_0,w)-k_D(z_0,f(w))].
\]
\end{definition}

Since
\begin{equation}\label{magzero}
\begin{split}
k_D(z_0,w)-k_D(z_0,f(w))&\geq k_D(f(z_0),f(w))-k_D(z_0,f(w))\\&\geq - k_D(f(z_0), z_0) >-\infty,
\end{split}
\end{equation}
the boundary dilation coefficient is always strictly positive. Moreover, the boundary dilation coefficient does not depend on $z_0$ and can be computed  using pluripotential theory   as in the classical Julia's lemma for the unit disc (see \cite{BCD}).

We state here the part of the Julia-Wolff-Carath\'eodory type theorem for strongly convex domains we need for our aims (see \cite[Thm. 2.7.14]{A}, \cite{Ab1} for the first part, and \cite[Prop. 3.4, Remark. 3.5]{B1} for (2)):

\begin{theorem}\label{JWC}
Let $D\subset \C^N$ be a bounded strongly convex   domain with smooth boundary and let $z_0\in D$ and $p\in \de D$.   Let $f\colon  D \to D$ be holomorphic and assume $\al_p(f)<+\infty$. Then  $f(p)=\Klim_{z\to p}f(z)$ exists and $f(p)\in \de D$. Moreover, \begin{enumerate}
  \item the holomorphic function
\[
D\ni z\mapsto \frac{1-\tilde{\rho}_{f(p)}(f(z))}{1-\tilde{\rho}_p(z)}
\]
has limit $\al_p(f)$ along any special and restricted sequence $\{z_k\}\subset D$ which converges to $p$. In particular, the holomorphic self-map of $\D$ given by $\zeta\mapsto \tr_{f(p)}(f(\v_p(\zeta)))$ has a boundary regular fixed point at $1$ and $\al_1\bigl(\tr_{f(p)}\circ f \circ \v_p\bigr)=\al_p(f)$.
  \item The curve $[0,1]\ni t\mapsto f(\v_p(t))$ converges to $f(p)$ and it is special and restricted.
\end{enumerate}
\end{theorem}

Boundary dilation coefficients satisfy the chain rule:

\begin{lemma}\label{chainrule}
Let  $D\subset \C^N$ be a bounded strongly convex   domain with smooth boundary. Let $f,g \colon  D \to D$ be holomorphic and let $p\in \de D$. If $\al_p(g\circ f)<+\infty$, then $\al_p(f)<+\infty$ and $\al_{f(p)}(g)<+\infty$, where $f(p)=\Klim_{z\to p}f(z)$. Moreover,
\begin{equation}\label{chain-eq}
\al_p(g\circ f)=\al_{f(p)}(g)\cdot \al_p(f).
\end{equation}
Conversely, if $\al_p(f)<+\infty$, then \eqref{chain-eq} holds.
\end{lemma}

\begin{proof} Let $z_0\in D$. Assume $\al_p(g\circ f)<+\infty$. Since
\begin{equation*}
\begin{split}
\al_p(g\circ f)&=\liminf_{w\to p} [k_D(z_0, w)- k_D(z_0, g(f(w)))]\\ &\geq \liminf_{w\to p} [k_D(z_0, w)- k_D(z_0, f(w))]\\&+\liminf_{w\to p} [k_D(z_0, f(w))- k_D(z_0, g(f(w)))]
\end{split}
\end{equation*}
and both terms in the right hand side are not $-\infty$ by \eqref{magzero}, it follows that $\al_p(f)<+\infty$. By Theorem \ref{JWC}, $f(p)$ exists and $\al_{f(p)}(g)<+\infty$. Moreover, for $t\in [0,1)$
\begin{equation}\label{divid}
\frac{1-\tilde{\rho}_{g(f(p))}(g(f(\v_p(t))))}{1-t}=
\frac{1-\tilde{\rho}_{g(f(p))}(g(f(\v_p(t))))}{1-\tilde{\rho}_{f(p)}(f(\v_p(t)))}\frac{1-\tilde{\rho}_{f(p)}(f(\v_p(t)))}{1-t}.
\end{equation}
Since $\tilde{\rho}_p(\v_p(t))=t$, equation \eqref{chain-eq} follows by Theorem \ref{JWC} letting $t\to 1$.

Conversely, assume $\al_p(f)<+\infty$. If $\al_{f(p)}(g)=+\infty$ then $\al_{p}(g\circ f)=+\infty$ for what we already proved. So we can assume $\al_{f(p)}(g)<+\infty$. By \eqref{divid} and Theorem \ref{JWC} the term on the left hand side has limit $\al_{f(p)}(g)\cdot \al_p(f)$ for $t\to 1$. By \cite[Thm. 2.7]{BCD} it follows that $\al_{p}(g\circ f)<+\infty$ and in fact \eqref{chain-eq} holds.
\end{proof}

\begin{definition}
Let $D\subset \C^N$ be a bounded strongly convex domain with smooth boundary. If $f\colon D \to D$ is holomorphic, we say that $p\in \de D$ is a {\sl contact point} for $f$ if $\Klim_{z\to p}f(z)=f(p)$ exists and $f(p)\in \de D$. The contact point $p$ is a {\sl regular contact point} if  $\al_f(p)<+\infty$.

In case $p=f(p)$, the point $p$ is called a {\sl boundary fixed point}. If it is also regular, it is called a {\sl boundary regular fixed point}---or BRFP for short. For $A>0$ we denote by
\[
\br_A(f):=\{p\in \de D \mid f(p)=p, \al_f(p)\leq A\}.
\]
A point $p\in \br_A(f)\setminus\br_1(f)$ is called a {\sl boundary repelling fixed point}.
\end{definition}

A point $p\in \br_A(f)$ is called {\sl isolated} if there exists a neighborhood $U$ of $p$ such that $\br_A(f)\cap U=\{p\}$.

In the unit disc, as a consequence of Cowen-Pommerenke's estimates \cite{CoPo} (see also \cite[Thm. 2.2]{B}), every boundary repelling fixed point is isolated. In higher dimension, this is no longer true (see \cite[Example 6.3]{O}).

The Denjoy-Wolff type theorem for strongly convex domains (see \cite[Thm. 2.4.23]{A}, \cite[Prop. 2.9]{BCD} and \cite{AR2}) can be stated as follows:

\begin{theorem}\label{Wolff}
Let $D\subset \C^N$ be a bounded strongly convex domain with smooth boundary and let $f\colon  D \to D$ be holomorphic. Then either
\begin{enumerate}
 \item there exists a complex geodesic $\v\colon \D\to D$ such that  $f(\v(\zeta))=\v(e^{i\theta}\zeta)$ for some $\theta\in \R$, and in particular $f(\v(0))=\v(0)$ and $\al_{f}(p)=1$ for all $p\in \de \v(\D)$, or
  \item there exists $x\in D$ such that $f(x)=x$ and for every $p\in \de D$ it holds $\al_f(p)>1$ --- and in such a case the sequence of iterates $\{f^{\circ m}\}$ converges uniformly on compacta to the constant map $z\mapsto x$, or
  \item $\br_1(f)$ contains a unique point $\tau\in \de D$ such that the sequence of iterates $\{f^{\circ m}\}$ converges uniformly on compacta to the constant map $z\mapsto \tau$ and $\al_f(p)>1$ for all $p\in \de D\setminus\{\tau\}$.
\end{enumerate}
\end{theorem}

A holomorphic self-map $f$ of $D$ is called {\sl rotational elliptic} if it satisfies (1) of Theorem \ref{Wolff}. It  is called {\sl strongly elliptic} if it satisfies (2) of Theorem \ref{Wolff}. Finally, $f$ is called {\sl non-elliptic} if it satisfies (3) of Theorem \ref{Wolff}.

In the non-elliptic case, the point $\tau$ is called the {\sl Denjoy-Wolff point} of $f$, and $f$ has no fixed points in $D$. In the strongly elliptic case, $f$ has a unique fixed point in $D$.

{\sl Boundary regular stationary points}, namely boundary regular fixed points with dilation coefficients $=1$ are very special. The following lemma follows at once from Theorem \ref{Wolff} and \cite[Prop. 2.9.(1)]{BCD}:

\begin{lemma}\label{stationary}
Let $D\subset \C^N$ be a bounded strongly convex domain with smooth boundary and let $f\colon  D \to D$ be holomorphic. Let $p\in \de D$ be such that $\al_f(p)=1$. If $p$ is isolated then $f$ is non-elliptic and $p$ is in fact the Denjoy-Wolff point of $f$.
\end{lemma}

\section{Backward iteration sequences}\label{Sback}

In order to prove our main result, we shall use the so-called backward iteration sequences.

\begin{definition}
Let $D\subset \C^N$ be a bounded strongly convex domain with smooth boundary and let $f\colon  D \to D$ be holomorphic. A {\sl backward iteration sequence} for $f$ at $p\in \de D$ is a sequence $\{w_k\}\subset D$ such that $f(w_{k+1})=w_k$ for all $k\in \N$ and
\[
\frac{1}{2}\log s(\{w_k\}):=\sup_{k\in \N} k_D(w_k, w_{k+1})<+\infty.
\]
The number $s(\{w_k\})$ is called the {\sl hyperbolic step} of the sequence $\{w_k\}$.
\end{definition}

Backward iteration sequences in the unit disc have been introduced in \cite{B}, exploiting results from \cite{PC}, in order to study BRFP's for commuting mappings, and they have been throughly studied in \cite{PC2}. Such results have been generalized to the ball in \cite{O} and to strongly convex domains for non rotational elliptic maps in \cite{AR1}. The following lemma is the content of \cite[Lemma 2.2, Lemma 2.3]{AR1}

\begin{lemma}\label{back-boundary}
Let $D\subset \C^N$ be a bounded strongly convex domain with smooth boundary, and let $f\colon  D \to D$ be a holomorphic map. Let $\{w_n\}$ be a backward iteration sequence converging toward the boundary of $D$. Then there exists a BRFP $p\in \de D$ for $f$ such that $w_n\to p$ and $\al_f(p)\leq s(\{w_k\})$.
\end{lemma}

The following result is proved in \cite[Thm. 0.1, Lemma 2.3, Thm. 3.3]{AR1}):

\begin{theorem}\label{backward}
Let $D\subset \C^N$ be a bounded strongly convex domain with smooth boundary, and let $f\colon  D \to D$ be a holomorphic map either strongly elliptic or non-elliptic.
\begin{enumerate}
  \item If $\{w_k\}$ is a backward iteration sequence  then $\{z_k\}$ converges to $p\in \br_A(f)$ for some $A\geq 1$ and $\al_f(p)\leq s(\{w_k\})$. In case $\al_f(p)=1$ then $f$ is non-elliptic and $p$ is the Denjoy-Wolff point of $f$.
  \item If $p\in \de D$ is an isolated boundary repelling fixed point of $f$,  then there exists a backward iteration sequence $\{w_k\}$ converging to $p$ with hyperbolic step $s(\{w_k\})\leq \al_f(p)$.
  \item If $\{w_k\}$ is a backward iteration sequence which converges to $p$ and $\al_f(p)>1$, then there exists $M>1$ and $k_0\in \N$ such that $\{w_k\}_{k\geq k_0}\subset K_{z_0}(p,M)$.
\end{enumerate}
\end{theorem}

We examine now the rotational elliptic case. In order to state our result, we need some notations. Let $D\subset \C^N$ be a bounded strongly convex domain with smooth boundary and let $f\colon D\to D$ be a holomorphic rotational elliptic map. By \cite[Thm. 2.1.29]{A} there exists a closed complex submanifold $M\subset D$ and a holomorphic map $\pi\colon  D \to D$ such that $\pi(D)=M$, $\pi\circ \pi=\pi$, $\pi\circ f=f\circ\pi$, $f(M)=M$ and $f|_M$ is an automorphism of $M$. The manifold $M$ is called the {\sl limit manifold} of $f$ since for all $z\in D\setminus M$ the limit set of $\{f^{\circ m}(z)\}$ belongs to $M$.
In particular, if $Z=\{z\in D \mid f(z)=z\}$, it follows that $\emptyset\neq Z\subseteq M$. The map $\pi$ is called the {\sl limit retraction} of~$f$, and it can be obtained as limit of a sequence of iterates of~$f$; and moreover if $\pi'\colon D\to D$ is another limit of a sequence of iterates of~$f$ such that $\pi'\circ\pi'=\pi'$ then $\pi'=\pi$. In particular, all iterates of $f$ have the same limit retraction and the same limit manifold.

\begin{proposition}\label{back-elliptic}
Let $D\subset \C^N$ be a bounded strongly convex domain with smooth boundary, and let $f\colon  D \to D$ be a rotational elliptic holomorphic map and let $M$ be the limit manifold of $f$.
\begin{enumerate}
\item A point $p\in \de D$  is a regular contact point of $f$ with $\al_f(p)=1$ if and only if $p\in \overline{M}\cap\de D$.
  \item If $\{w_k\}\subset D$ is a backward iteration sequence and there exists $R>0$ such that $\inf_{y\in M}\|w_n-y\|\geq R$  then $\{w_k\}$ converges to $p\in \br_A(f)$ for some $A> 1$   and $s(\{z_k\})\geq \al_f(p)>1$.
  \item If $p\in \de D$ is an isolated boundary repelling fixed point of $f$,  then there exists a backward iteration sequence $\{w_k\}$ converging to $p$ with hyperbolic step $s(\{w_k\})\leq \al_f(p)$.
  \item If $\{w_k\}$ is a backward iteration sequence which converges to $p$ and $\al_f(p)>1$, then there exists $M>1$ and $k_0\in \N$ such that $\{w_k\}_{k\geq k_0}\subset K_{z_0}(p,M)$.
\end{enumerate}
\end{proposition}

\begin{proof}
(1) Let $p\in \de D\cap \overline{M}$. Let $z_0\in M$ be such that $f(z_0)=z_0$, and let $\v_p\colon \D \to D$ be the complex geodesic such that $\v_p(0)=z_0$, $\v_p(1)=p$.

We claim that $\v_p(\D)\subset M$. Since $\pi\circ \pi=\pi$ and by the decreasing property of the Kobayashi distance, for all $z,w\in M$
\[
k_M(z,w)= k_M(\pi(z),\pi(w))\leq k_D(z,w)\leq k_M(z,w),
\]
from which it follows that $M$ is totally geodesic in $D$, {\sl i.e.}, $k_D|_M=k_M$. In particular, if $\eta\colon \D \to D$ is a complex geodesic such that $\eta(0)\in M$ and $\eta(r)\in M$ for some $r\in (0,1)$, we have
\[
k_D(\pi(\eta(0)), \pi(\eta(r)))=k_D(\eta(0),\eta(r))=k_{\D}(0,r),
\]
and hence $\pi \circ \eta\colon  \D \to D$ is also a complex geodesic. By the uniqueness of complex geodesics up to pre-composition with automorphisms of $\D$, it follows that in fact $\pi\circ \eta(\zeta)=\eta(\zeta)$ for all $\zeta\in \D$. Thus, $\eta(\D)\subset M$ and $\eta$ is  a complex geodesic both for $D$ and for $M$. Now, let $\{w_m\}\subset M$ be a sequence which converges to $p$ and let $\v_m\colon \D\to D$ be the complex geodesic such that $\v_m(0)=z_0$ and $\v_m(r_m)=w_m$ for some $r_m\in (0,1)$. For what we just proved, $\v_m(\D)\subset M$. Up to subsequences, we can assume that $\{\v_m\}$ converges uniformly on compacta to a holomorphic map $h\colon \D \to M$ such that $h(0)=z_0$. By Proposition \ref{geo-converge}, $\v_m\to h$ uniformly on $\oD$ and $h$ is a complex geodesic for   $D$ (and for $M$). But, since $\v_m(r_m)\to p$, it follows that $h(1)=p$ and hence $h=\v_p$, which proves that $\v_p(\D)\subset M$.

Now, since $f\colon M\to M$ is an automorphism, and in particular an isometry for the Kobayashi distance, $f\circ \v_p\colon \D \to M$ is a complex geodesic in $M$, and hence in $D$. In particular, $\lim_{r\to 1}f(\v_p(r))=q\in \de D$ exists and $\v_q\equiv f\circ \v_p$. Therefore,  $\tr_q(f(\v_p(\zeta)))=\zeta$ for all $\zeta\in \D$ and by Theorem \ref{JWC}.(1), $\al_f(p)=1$.

Conversely, assume $p\in \de D$ is a regular contact point with $\al_f(p)=1$. Let $q=f(p)$ and $z_0\in Z$. Let $\v_p, \v_q$ be the complex geodesics such that $\v_p(0)=\v_q(0)=z_0$ and $\v_p(1)=p$, $\v_q(1)=q$. Then $g:=\tr_q\circ f\circ \v_p$ is a holomorphic self-map of $\D$ such that $g(0)=0$ and $g(1)=1$, $\al_g(1)=\al_f(p)=1$ by Theorem \ref{JWC}. As a consequence of the classical Julia's lemma \cite[Corollary 1.2.10]{A}, $g(\zeta)=\zeta$ for all $\zeta\in \D$. Hence for all $\zeta_1,\zeta_2\in \D$
\[
k_{\D}(\zeta_1,\zeta_2)\leq k_D(f(\v_p(\zeta_1)),f(\v_p(\zeta_1)))\leq k_D(\v_p(\zeta_1), \v_p(\zeta_2))=k_{\D}(\zeta_1,\zeta_2),
\]
and thus $f\circ \v_p$ is a complex geodesic. The result follows then from:

\smallskip
{\bf Claim A:} if $\v\colon \D \to D$ is a complex geodesic such that $\v(0)=z_0\in Z$ and $f\circ \v\colon  \D \to D$ is also a complex geodesic, then $\v(\D)\subset M$.
\smallskip

In order to prove Claim 1., recall that by \cite[Thm. 2.1.21]{A}, the tangent space to $D$ at $z_0$ admits a $df_{z_0}$-invariant splitting $T_{z_0}D=L_N\oplus L_U$ such that $T_{z_0}M=L_U$ and the spectrum of $df_{z_0}$ in $L_U$ is contained in $\de \D$, while the spectrum of $df_{z_0}$ in $L_N$ is contained in $\D$. Note that   $d\pi_{z_0}|_{L_U}={\sf id}$. Therefore  if $\eta\colon \D \to D$ is a complex geodesic such that $\eta(0)=z_0$ and $\eta'(0)\in L_U$, then $\pi\circ \eta\colon \D \to D$ is a holomorphic map such that $\kappa_D(z_0; (\pi\circ \eta)'(0))=\kappa_{\D}(0; 1)$, hence it is an infinitesimal isometry for the Kobayashi metric. Thus $\pi\circ \eta$ is a complex complex, and by the uniqueness of infinitesimal isometry, $\pi\circ \eta=\eta$. Hence $\eta(\D)\subset M$.

In the hypothesis of Claim A,   $\kappa_D(z_0; \v'(0))=\kappa_D(z_0; df_{z_0}(\v'(0))$. Hence $\v'(0)$ belongs to $L_U$ and therefore $\v(\D)\subset M$.

(2) {\sl Step 1.} Let $z_0\in Z$. We claim that
\begin{equation}\label{stimaM}
k_D(z,z_0)>k_D(f(z), z_0)\quad \forall z\in D\setminus M.
\end{equation}
Clearly, $k_D(f(z), z_0)\leq k_D(z,z_0)$. Assume by contradiction that $k_D(f(z), z_0)=k_D(z,z_0)$ for some $z\in D\setminus M$. Let $\v\colon  \D \to D$ be the complex geodesic such that $\v(0)=z_0$ and $\v(r)=z$ for some $r\in (0,1)$. Then
\[
k_D(f(\v(r)),\v(0))=k_D(f(z), z_0)=k_D(z,z_0)=k_D(\v(r), \v(0)),
\]
and therefore $f\circ \v$ is a complex geodesic in $D$. By Claim A, $\v(\D)\subset M$, thus $z\in M$, a contradiction.

{\sl Step 2.} Let $z_0\in Z$.
\medskip

{\bf Claim B:} For all $R_0>0$ there exists $0<c=c(R_0)<1$ such that for all $z\in D$ with $\inf_{y\in M}\|z-y\|\geq R_0$,  it holds
\begin{equation}\label{stimac}
k_D(f(z), z_0)-k_D(z,z_0)\leq \frac{1}{2}\log c<0.
\end{equation}

\medskip
Assume \eqref{stimac} is not true. Then for every $c<1$ there exists $z(c)\in D$ such that $\inf_{y\in M}\|z(c)-y\|\geq R_0$ and
\begin{equation}\label{dist-mm}
k_D(f(z(c)), z_0)-k_D(z(c), z_0)> \frac{1}{2}\log c.
\end{equation}
Let $x\in \overline{D}$ be a limit point of $\{z(1-\frac{1}{n})\}$. Since $\inf_{y\in M}\|z(c)-y\|\geq R_0$, it follows that $x\not\in\overline{M}$. If $x\in D$, it follows from \eqref{dist-mm} that $k_D(f(z(c)), z_0)\geq k_D(z(c), z_0)$, contradicting \eqref{stimaM}. Hence $x\in \de D$ and
\[
\liminf_{z\to x}[k_D(z,z_0)-k_D(f(z),z_0)]\leq 0.
\]
But then $x$ is a boundary regular contact point with $\al_f(x)\leq 1$, and by part (1), $x\in \overline{M}$, again a contradiction.

{\sl Step 3.} Let $\{w_k\}$ be a backward iteration sequence satisfying the hypothesis of (2). Let $c=c(R)$ be given by Step 2. By induction on \eqref{stimac} we obtain that for all $k\in \N$,
\[
e^{-2k_D(w_k,z_0)}\leq c^k e^{-2k_D(w_0,z_0)},
\]
hence $k_D(w_k,z_0)\to \infty$ for $k\to \infty$. Therefore, $w_k\to \de D$. The result follows then from Lemma \ref{back-boundary}.

(3) By (1), $p\not\in \overline{M}$. Let $\epsilon>0$ be such that  $U:=\{z\in \C^N\mid \|z-p\|<\epsilon\}$ has the property that $\overline{U}\cap \overline{M}=\emptyset$ and $U\cap \br_{\al_f(p)}(f)=\{p\}$. Arguing exactly  as in the proof of \cite[Thm. 3.3]{AR1} we can construct a backward iteration sequence $\{w_k\}\subset U$ with  hyperbolic step $s(\{w_k\})\leq \al_f(p)$. Since $\overline{U}\cap \overline{M}=\emptyset$, $\{w_k\}$ satisfies the hypothesis of (2), hence it converges to a boundary regular fixed point of $f$, say $q\in \de D\cap U$, with $\al_f(q)\leq \al_f(p)$. But $U\cap \br_{\al_f(p)}(f)=\{p\}$ implies $q=p$ and we are done.

(4) Since $\{w_k\}$ converges to a boundary repelling fixed point $p\in \de D$, and by (1), $p\not\in\overline{M}$, hence $\{w_k\}$ satisfies the hypothesis of Claim B. Therefore, by \eqref{stimac}
\[
\liminf_{k\to \infty}[k_D(z_0, z_{k+1})-k_D(z_0,z_k)]\geq \frac{1}{2}\log \frac{1}{c}>0.
\]
Now the proof follows arguing exactly as in \cite[Lemma 2.5]{AR1}.

\end{proof}

\section{Semigroups}

Let $D\subset \C^N$ be a bounded strongly convex domain with smooth boundary. A {\sl (continuous) semigroup} $(\phi_t)$ of holomorphic self-maps of $D$ is a continuous homomorphism from the semigroup $(\R^+,+)$ endowed with the Euclidean topology, to the semigroup (with respect of maps composition) of holomorphic self-maps of $D$ endowed with the topology of uniform convergence on compacta (see, {\sl e.g.}, \cite{RS}, \cite{A}, \cite{BCD}).

If $(\phi_t)$ is a semigroup of holomorphic self-maps of $D$, we will denote by $\al_t(p):=\al_p(\phi_t)$ the dilation coefficient of $\phi_t$ at $p\in \de D$.

Also, we denote by ${\sf Fix}(\phi_t):=\{z\in D\mid \phi_t(z)=z\ \forall t\geq 0\}$.

The following result was proved in \cite[Thm. 2.5.24, Prop. 2.5.26]{A}; see also \cite{AS} and \cite[Thm. A.1]{BCD1}.

\begin{theorem}\label{densem}
Let $D\subset \C^n$ be a bounded strongly convex domain with smooth
boundary. Let $(\phi_t)$ be a  semigroup of
holomorphic self-maps of $D$. Then either
\begin{itemize}
\item ${\sf Fix}(\phi_t)\neq \emptyset$, or
\item $\phi_t$ is non-elliptic for all $t>0$ and there exists a unique $\tau \in \partial D$ such that
$\tau$ is the Denjoy-Wolff point of $\phi_t$.
\end{itemize}
\end{theorem}

The following result follows from \cite[Corollary 4.8, Prop. 3.3]{BCD}:

\begin{proposition}\label{variesem}
Let $D\subset \C^n$ be a bounded strongly convex domain with smooth
boundary.  Let $(\phi_t)$ be a  semigroup of holomorphic self-maps of $D$.
\begin{enumerate}
  \item Suppose there exists $t_0>0$ such that $\phi_{t_0}$ is strongly elliptic. Then $\phi_t$ is  strongly elliptic for all $t\geq 0$.
  \item If $p\in \de D$ is a BRFP for $\phi_t$ for every $t\geq 0$, then there exists $\lambda\in (0,+\infty)$ such that $\al_t(p)=\lambda^t$.
\end{enumerate}
\end{proposition}

The previous results ensure that the following definition is well-posed:

\begin{definition}
Let $D\subset \C^n$ be a bounded strongly convex domain with smooth boundary. Let $(\phi_t)$ be a  semigroup of holomorphic self-maps of $D$.
\begin{enumerate}
  \item $(\phi_t)$ is {\sl non-elliptic} if $\phi_1$ is non-elliptic. In such a case, if $\tau\in \de D$ is the Denjoy-Wolff point of $\phi_1$, we call $\tau$ the {\sl Denjoy-Wolff point} of $(\phi_t)$.
  \item $(\phi_t)$ is {\sl  strongly elliptic} if $\phi_1$ is  strongly elliptic.
  \item $(\phi_t)$ is {\sl  rotational elliptic} if $\phi_1$ is  rotational elliptic.
\end{enumerate}
\end{definition}

For rotational elliptic semigroups, $(\phi_t)$, let $M_t$ denote the limit manifold of $\phi_t$ for $t>0$. We let
\[
M(\phi_t):=M_1.
\]

\begin{lemma}\label{limmansem}
Let $D\subset \C^n$ be a bounded strongly convex domain with smooth boundary. Let $(\phi_t)$ be a rotational elliptic semigroup of holomorphic self-maps of~$D$. Then $M_t=M(\phi_t)$ for every $t>0$. In particular, $(\phi_t)|_{M(\phi_t)}$ is a group of automorphisms of $M(\phi_t)$.
\end{lemma}

\begin{proof}
Since $\phi_1$ is an iterate of $\phi_{1/q}$ for any $q\ge 1$, and the limit manifold of an iterate coincides with the limit manifold of the original map, we have $M_{1/q}=M_1$ for all $q\ge 1$.
For the same reason we have $M_{p/q}=M_{1/q}=M_1$ for all $p/q\in\Q^+$. Since $\Q^+$ is dense in~$\R^+$ it follows that $M_t= M_1$ for all $t>0$.
\end{proof}

\section{Regular contact points for semigroups}\label{Scontact}

Contact points for semigroups of the unit disc were studied in \cite{BP} exploiting the existence of the so-called K\"onigs' function, a tool which is not available in higher dimension. To replace it, we shall use complex geodesics.

\begin{proposition}\label{contatto}
Let $D\subset \C^N$ be a bounded strongly convex domain with smooth boundary. Let $(\phi_t)$ be a semigroup of holomorphic self-maps of $D$. Let $t_0>0$ and let $p\in \de D$ be  a regular contact point for $\phi_{t_0}$. Then there exists $T\in [t_0,+\infty]$ such that $p$ is a regular contact point for $\phi_t$ for all $t\in [0,T)$. Moreover, the curve $[0,T)\ni r\mapsto \phi_r(p)\in \de D$ is continuous.
\end{proposition}

\begin{proof} If $p$ is  the Denjoy-Wolff point of $(\phi_t)$ (when $(\phi_t)$ is non-elliptic) or it belongs to the closure of ${\sf Fix}(\phi_t)$ (when $(\phi_t)$ is elliptic but not strongly elliptic) then $p$ is a BRFP for all $\phi_t$ by Theorem \ref{Wolff} and the statement is true. Thus, we can suppose that $p$ is neither the Denjoy-Wolff nor in the closure of ${\sf Fix}(\phi_t)$.

Let $T:=\sup {t\in [0,+\infty)}$ such that $p$ is a regular contact point for $\phi_t$. By assumption $T\geq t_0$.

Fix $\epsilon >0$. If $T<+\infty$,  let $T\geq t_1 >T-\epsilon$ be such that $p$ is a regular contact point for $\phi_{t_1}$. If $T=+\infty$, let $t_1\geq 1/\epsilon$ be such that $p$ is a regular contact point for $\phi_{t_1}$. We will show that for all $t\in [0,t_1]$, $p$ is a regular contact point for $\phi_t$. Taking $\epsilon\to 0$, we will get the first statement.

To this aim, let $s$,~$t> 0$ be such that $s+t=t_1$. Since $\phi_{t_1}=\phi_{t+s}=\phi_t\circ \phi_s$, it follows by Lemma \ref{chainrule} and Theorem \ref{JWC} that $p$ is a regular contact point for $\phi_s$, that $\phi_s(p)\in \de D$ is a regular contact point for $\phi_t$ and that
\begin{equation}\label{estimo}
\al_t\bigl(\phi_s(p)\bigr)\cdot \al_s(p)=\al_{t+s}(p)=\al_{t_1}(p)<+\infty\;.
\end{equation}
In particular, it follows that $p$ is a regular contact point for $\phi_t$ for all $t\in [0,T)$.

Now, we consider the curve $[0,T)\ni r\mapsto \phi_r(p)\in \de D$ and we prove that it is continuous. Fix $z_0\in D$. Let  $\v\colon \D \to D$ be the complex geodesic such that $\v(0)=z_0$ and $\v(1)=p$. Let $r_0\in [0,T)$ be fixed. For $r\in [0,T)$ close to $r_0$, let $\v_r\colon \D \to D$ denote the complex geodesic such that $\v_r(0)=z_0$ and $\v_r(1)=\phi_r(p)$ and let $\tr_r\colon D \to \D$ be the left-inverse of $\v_r$. Let define
\[
g_r(\zeta):=\tr_r(\phi_r(\v(\zeta)))\quad \forall \zeta\in \D.
\]
By construction, $g_r\colon \D\to \D$ is holomorphic and by Theorem \ref{JWC}.(1), $1$ is a boundary regular fixed point for $g_r$ with boundary dilation coefficient $\al_r(p)$. Thus, again by the classical Julia-Wolff-Carath\'eodory theorem in $\D$, the derivative $g'_r$ has non-tangential limit $\al_r(p)$ at $1$.

Let fix $t_1<T$ such that $t_1>r_0$. From \eqref{estimo} we get that for all $r\leq t_1$ setting $t=t_1-r$ it holds
$\al_r(p)=\al_{t_1}(p)/\al_t(\phi_r(p))$. Now, if $\al_t(\phi_r(p))\geq 1$ for all $r\in[0,t_1]$,  it follows that $\al_r(p)\leq \al_{t_1}(p)$ for all $r\in [0,t_1]$. On the other hand, if $\al_t(\phi_{\tilde r}(p))< 1$ for some $\tilde r\le t_1$, by Theorem \ref{Wolff} it follows that $q=\phi_{\tilde r}(p)$ is  the Denjoy-Wolff point of $(\phi_t)$; in particular, by  Proposition \ref{variesem}, $\al_u(\phi_{\tilde r}(p))=e^{\beta u}$ for all $u\ge 0$ and some $\beta<0$ independent of~$\tilde r$. It follows that $\al_{\tilde r}(p)\leq \al_{t_1}(p)e^{-\beta t_1}$, and thus we have proved that there exists $C>0$ such that
\[
\al_r(p)\leq C
\]
for all $r\leq t_1$.

Let  $M>1$ and let
\[
\overline{K}:=\bigl\{\zeta \in \oD\bigm| |1-\zeta|\le M(1-|\zeta|)\bigr\} \subset \D\cup\{1\}
\]
be (the closure of) a Stolz angle in $\D$ with vertex $1$ (see, {\sl e.g.}, \cite[pag. 53]{A}). Since $g'_r(1)=\al_r(p)\leq C$ for all $r\leq t_1$, it follows that $\{g_r'\}$ is equibounded in $\overline{K}$. Hence $\{g_r\}$ is equicontinuous on $\overline{K}$ (and it is clearly  equibounded in $\overline{K}$ by $1$). Applying Ascoli-Arzel\`a theorem, we find a subsequence $r_k\to r_0$ such that $\{g_{r_k}\}$ converges uniformly on $\overline{K}$ to some continuous function $g$. In particular, note that $g(1)=1$.

Up to a subsequence, we can also assume that $\{\v_{r_k}\}$ converges uniformly on compacta of $D$ to a holomorphic map $\eta\colon \D \to D$ such that $\eta(0)=z_0$. By Proposition \ref{geo-converge}, $\eta$ is a complex geodesic and  $\v_{r_k}\to \eta$ uniformly on $\oD$. Since $\v_{r_k}(1)=\phi_{r_k}(p)$ by construction,  it follows that $\phi_{r_k}(p)\to \eta(1)$. Let $\tr\colon  D \to \D$ be the left inverse of~$\eta$.  By Proposition \ref{geo-converge}, $\{\tr_{r_k}\}$ converges uniformly on $\overline{D}$ to $\tr$.

Since $\phi_{r_k}\to \phi_{r_0}$ uniformly on compacta of $D$, it follows that $g(\zeta)=\tr(\phi_{r_0}(\v(\zeta)))$ for all $\zeta \in K$. Taking the radial limit at $1$, we obtain
\[
1=\lim_{(0,1)\ni s\to 1}g(s)=\lim_{(0,1)\ni s\to 1}\tr(\phi_{r_0}(\v(s)))=\tr(\phi_{r_0}(p)).
\]
Since the only point in the fiber of $\tr$ over $1$ is the point $\eta(1)$, it follows that $\eta(1)=\phi_{r_0}(p)$, hence $\phi_{r_k}(p)\to \phi_{r_0}(p)$. Repeating the argument for any subsequence, we obtain the result.
\end{proof}

\section{Common BRFP's for semigroups}\label{prova}

\begin{theorem}\label{comune}
Let $D\subset \C^N$ be a bounded strongly convex domain with smooth boundary. Let $(\phi_t)$ be a semigroup of holomorphic self-maps of $D$. Suppose $p\in \de D$    is an isolated boundary repelling fixed point for $\phi_{t_0}$ for some $t_0>0$. Then $p$ is a boundary repelling fixed point for $\phi_t$ for all $t> 0$.
\end{theorem}

\begin{proof}
By Theorem \ref{backward} or Proposition \ref{back-elliptic} there exists a backward iteration sequence $\{w_n\}$ for $\phi_{t_0}$ such that $k_D(w_n,w_{n+1})\leq \frac{1}{2}\log \al_{t_0}(p)$ and $\{w_n\}$ converges to $p$ inside a $K$-region. For $t\in [0,t_0)$, define $z_n^t:=\phi_t(w_n)$. Then
\[
\phi_{t_0}(z_n^t)=\phi_{t}(\phi_{t_0}(w_n))=\phi_t(w_{n-1})=z_{n-1}^t,
\]
and
\begin{equation}\label{cast}
k_D(z_n^t, z_{n+1}^t)=k_D(\phi_t(w_n), \phi_t(w_{n+1}))\leq k_D(w_n,w_{n+1})\leq \frac{1}{2}\log \al_{t_0}(p).
\end{equation}
Hence $\{z_n^t\}$ is a backward iteration sequence for $\phi_{t_0}$. Moreover, by Proposition \ref{contatto}, $p$ is a regular contact point for $\phi_t$ all $t\in [0, t_0]$. Since  $\{w_n\}$ converges to $p$ inside a $K$-region, by Theorem \ref{JWC}
\[
q_t:=\lim_{n\to \infty} z_n^t=\lim_{n\to\infty}\phi_t(w_n)=\phi_t(p)\in \de D.
\]
Therefore, by Lemma \ref{back-boundary}, $q_t$ is a BRFP  for $\phi_{t_0}$ with dilation coefficient $\al_{t_0}(q_t)\leq \al_{t_0}(p)$.

Hence, by Proposition \ref{contatto}, the curve $[0,t_0]\ni t\mapsto q_t$ is a continuous curve made of BRFP's of $\phi_{t_0}$ and $\al_{t_0}(q_t)\leq \al_{t_0}(p)$. Since $p$ is isolated, the only possibility is  $q_t=p$ for all $t\in [0,t_0]$.

Now, let $t>0$. Then $t=mt_0+s$ for some $m\in \N$ and $s\in [0, t_0)$. Hence
\[
\phi_{t}(p)=\phi_{mt_0+s}(p)=\phi_{mt_0}(\phi_s(p))=\phi_{t_0}^{\circ m}(p)=p.
\]
Moreover, by Lemma \ref{chainrule}, $\al_{t}(p)=\al_{t_0}(p)^m\cdot \al_s(p)<+\infty$, which implies that $p$ is a common BRFP for $(\phi_t)$.
\end{proof}

Now we can prove Theorem \ref{main}:

\begin{proof}[Proof of Theorem \ref{main}] By Theorem \ref{densem} and Lemma \ref{stationary}, if $\al_f(p)\leq 1$ then $p$ is the common Denjoy-Wolff point of $(\phi_t)$. If $\al_f(p)>1$ the result follows from Theorem \ref{comune}.
\end{proof}

For rotational elliptic semigroups there might exist boundary regular (in fact stationary) non isolated fixed points which are not fixed for all the elements of the semigroup:

\begin{example}
Let $\phi_t(z,w)=(e^{2\pi i t}z, w)$. Then $(\phi_t)$ is a rotational elliptic semigroup of $\B^2$. The points $\de \D \times \{0\}$ are BRFP's for $\phi_1$ (with boundary dilation coefficient $1$) but not for $\phi_t$ with $t\neq 0 \mod 1$.
\end{example}

\end{document}